\documentclass[12pt]{amsart}
\usepackage{amsmath}

\newcommand{\cB}{\mathcal{B}}

\newcommand{\alephes}{{\aleph_0}}

\newcommand{\op}{\operatorname}
\newcommand{\cf}{\op{cf}}

\newcommand{\sm}{\setminus}

\long\def\forget#1\forgotten{}

\newcommand{\Union}{\bigcup}
\newcommand{\sbst}{\subseteq}

\newcommand{\N}{\mathbb{N}}

\newcommand{\be}{\begin{enumerate}}
\newcommand{\ee}{\end{enumerate}}
\newcommand{\bi}{\begin{itemize}}
\newcommand{\ei}{\end{itemize}}
\renewcommand{\i}{\item}

\newtheorem{thm}{Theorem}
\newtheorem{prop}[thm]{Proposition}
\newtheorem{prob}[thm]{Problem}

\newtheorem{lem}[thm]{Lemma}

\theoremstyle{definition}
\newtheorem{defn}[thm]{Definition}

\theoremstyle{remark}

\title{On a problem of Juh\'asz and van Mill}
\author{Saharon Shelah}
\address{
Institute of Mathematics, Hebrew University of Jerusalem, Givat
Ram, 91904 Jerusalem, Israel, and Mathematics Department, Rutgers
University, New Brunswick, NJ 08903, U.S.A. }
\email{shelah@math.huji.ac.il}

\author{Boaz Tsaban}
\address{Department of Mathematics, Bar-Ilan University,
Ramat-Gan 52900, Israel;
and
Department of Mathematics,
Weizmann Institute of Science, Rehovot 76100, Israel.}
\email{tsaban@math.biu.ac.il}
\urladdr{http://www.cs.biu.ac.il/\~{}tsaban}

\begin{document}
\begin{abstract}
A $30$ years old and still open problem of Juh\'asz and van Mill
asks whether there exists a cardinal $\kappa$ such that
every regular dense in itself countably compact space
has a dense in itself subset of cardinality at most $\kappa$.
We give a negative answer for the analogous question
where \emph{regular} is weakened to \emph{Hausdorff},
and \emph{coutnably compact} is strengthened to
\emph{sequentially compact}.
\end{abstract}

\maketitle

\section{Introduction}

Every compact Hausdorff dense in itself topological space contains
a countable subset which is dense in itself \cite{JvM}.
A space is \emph{countably compact} if each countably infinite subset
of that space has a cluster point.
Tka\v{c}enko \cite{Tka} asked whether every regular
countably compact dense in itself topological space
contains a countable subset which is dense in itself. This was
answered in the negative by Juh\'asz and van Mill \cite{JvM}, who
concluded their paper with the following.

\begin{prob}[{Juh\'asz and van Mill \cite{JvM}}]\label{jvm}
Does there exist a cardinal $\kappa$ such that every
regular  dense in itself countably compact topological space
has a dense in itself subset of cardinality at most $\kappa$?
\end{prob}

Despite several efforts made in the past, the problem is still open.
We will prove the following. A topological space is \emph{scattered}
if each of its nonempty subsets has an isolated point.

\begin{thm}\label{main}
Assume that $\lambda^\alephes=\lambda$.
There exists a Hausdorff space $S$ such that:
\be
\i $|S|=\lambda$;
\i $S$ is dense in itself;
\i Each subspace of $S$ of cardinality less than the cofinality of $\lambda$ is scattered; and
\i $S$ is sequentially compact.
\ee
\end{thm}

The cardinals satisfying the assumption of the theorem are those
of the form $\lambda=\mu^\alephes$. Note that in this case,
$\cf(\lambda)>\alephes$. By the Hausdorff Formula, if
$\lambda^\alephes = \lambda$, then we have that
$$(\lambda^+)^\alephes=\max\{\lambda^+,\lambda^\alephes\} = \lambda^+,$$
so for each cardinal $\kappa$ we can take $\lambda=(\kappa^\alephes)^+$ (which is regular
and greater than $\kappa$)
to obtain a negative answer to Problem \ref{jvm} when stated for Hausdorff
spaces.

Of course, the original Problem \ref{jvm} is still interesting.
We hope that this paper will revive the interest in this problem.

\section{Proof of Theorem \ref{main}}

\begin{defn}
For a cardinal $\lambda$,
let $\lambda^*$ denote the collection of finite sequences of
elements of $\lambda$, and define a topology on $\lambda^*$
with subbase consisting of the \emph{$\epsilon$-neighborhoods}
$$[\eta]_\epsilon = \{\eta\}\cup\{\rho\in\lambda^* : \eta\subset\rho\mbox{ and }\rho(|\eta|)\ge\epsilon\},$$
of each $\eta\in\lambda^*$, where $\epsilon<\lambda$.
\end{defn}

\begin{lem}\label{lambda*}
\mbox{}
\be
\i $|\lambda^*|=\lambda$;
\i $\lambda^*$ is dense in itself;
\i For each $Y\sbst \lambda^*$ with $|Y|<\cf(\lambda)$:
\be
\i There exists $\epsilon<\lambda$ such that
the $\epsilon$-neighborhoods $[\eta]_\epsilon$, $\eta\in Y$, are pairwise disjoint.
\i $Y$ is closed and discrete.
\ee
\item Every two disjoint subsets of $\lambda^*$ of cardinality smaller than $\cf(\lambda)$
can be separated by disjoint open subsets of $\lambda^*$.
\ee
\end{lem}
\begin{proof}
(1) is immediate. To prove (2), note that each $\eta\in\lambda^*$ is the
accumulation point of $\{\eta\frown\epsilon : \epsilon<\lambda\}$.

(3a) We claim that for all
distinct $\nu,\eta\in\lambda^*$, there exists
$\epsilon_{\nu,\eta}<\lambda$ such that $[\eta]_{\epsilon_{\nu,\eta}}\cap [\nu]_{\epsilon_{\nu,\eta}}=\emptyset$.
Indeed, if $\nu$ and $\eta$ are incomparable with regards to $\sbst$, then we can take
$\epsilon_{\nu,\eta}=0$. Otherwise assume that $\nu\subset\eta$. Then we can take
$\epsilon_{\nu,\eta}=\eta(|\nu|)+1$.

Assume that $|Y|<\cf(\lambda)$.
Define $\epsilon = \sup\{\epsilon_{\nu,\eta} : \nu,\eta\in Y\}$. As $|Y|<\cf(\lambda)$,
$\epsilon<\lambda$ and is as required.

(3b) Assume that $|Y|<\cf(\lambda)$. By (3a), $Y$ is discrete.
To see that $Y$ is closed, assume that $y$ is in the closure of $Y$,
and apply (3a) to $Y\cup\{y\}$ to see that $y\in Y$.

(4) follow from (3a), either.
\end{proof}

The proof of the forthcoming Proposition \ref{mainlem} uses weak
bases. A \emph{weak base} of a topological space $X$ is a family of the form
$\Union_{x\in X}\cB(x)$ such that:
\be
\i For each $x\in X$:
\be
\i $x\in B$ for all $B\in \cB(x)$;
\i For all $B_1,B_2\in\cB(x)$, $B_1\cap B_2$ contains an element of $\cB(x)$;
\ee
\i For each $U\sbst X$, $U$ is open in $X$ if, and only if, for each $x\in U$, $U$ contains
an element of $\cB(x)$.
\ee
A space $X$ with a weak base $\Union_{x\in X}\cB(x)$ is \emph{weakly locally-countable/Haus\-dorff/regular} if, respectively:
\be
\i For each $x\in X$, there is a countable element in $\cB(x)$;
\i For all $x,y\in X$, there are disjoint $B_x\in\cB(x), B_y\in\cB(y)$;
\i $X$ is weakly Hausdorff and all elements of $\Union_{x\in X}\cB(x)$ are closed.
\ee

\begin{lem}[Nyikos-Vaughan \cite{NyikosVaughan92}]\label{alephesH}
For each weakly locally-countable space $X$:
\be
\i $X$ is locally countable; and
\i If $X$ is weakly regular, then it is Hausdorff.
\ee
\end{lem}
\begin{proof}
For the reader's convenience, we reproduce the proofs.

(1) Let $x\in X$. Set $U_0=\{x\}$ and inductively,
for each $y\in U_n$ choose a countable $D_y\in\cB(y)$,
and set $U_{n+1}=\Union_{y\in U_n}D_y$.
Then $\Union_nU_n$ is a countable open neighborhood of $x$.

(2) Let $x,y\in X$ be distinct. By (1), there is a countable open
neighborhood $U$ of $x,y$. Enumerate $U=\{x_n : n\in\N\}$.
Let $U_0=\{x\}, V_0=\{y\}$. By induction on $n$, assume
that $U_n,V_n$ are disjoint and closed. If $x_n\in U_n$,
pick $B\in\cB(x_n)$ such that $B\sbst U\sm V_n$ and let
$U_{n+1}=U_n\cup B, V_{n+1}=V_n$.
Otherwise, pick $B\in\cB(x_n)$ such that $B\sbst U\sm U_n$ and let
$U_{n+1}=U_n, V_{n+1}=V_n\cup B$.

The sets $U=\Union_n U_n, V=\Union_n V_n$ are disjoint open neighborhoods
of $x,y$, respectively.
\end{proof}

\begin{prop}\label{mainlem}
Assume that $\lambda^\alephes=\lambda$.
There exists a Hausdorff space $X$ such that:
\be
\i $|X|=\lambda$;
\i $X$ is locally countable;
\i $X$ is scattered;
\i There are $\lambda$ isolated points in $X$; and
\i $X$ is sequentially compact.
\ee
\end{prop}
\begin{proof}
This is proved as in Nyikos and Vaughan's ``Construction of $X_u$'' \cite[Page 313]{NyikosVaughan92}.
In our case, we should replace $\mathfrak{c}$ with $\lambda$ (and use the premise $\lambda^\alephes=\lambda$),
need not worry about the ultrafilter $u$, and verify that there are $\lambda$ isolated points in $X$.
We give a full direct proof.

The topology is constructed on the space $X=\lambda$.

Enumerate $[\lambda]^\alephes=\{D_\alpha : \omega\le\alpha<\lambda\}$,
such that $D_\alpha\sbst\alpha$ for each $\alpha$.
We construct, inductively on $\alpha$ with $\omega\le\alpha\le\lambda$,
topologies $\tau_\alpha$ on $\alpha$,
starting with $\alpha=\omega$ and taking the discrete topology $\tau_\omega$ on $\omega$.
(For each $n$, we set $\cB(n)=\{n\}$.)

\paragraph{When $\alpha$ is a limit ordinal:}
Let $\tau_\alpha$ be the topology having $\Union_{\beta<\alpha}\tau_\beta$
as a base.

\paragraph{When $\alpha=\beta+1$ is a successor ordinal:}
If $D_\beta$ has a cluster point in $(\beta,\tau_\beta)$, set
$\cB(\beta)=\{\{\beta\}\}$, and let $\tau_\alpha$ be the topology
having $\tau_\beta\cup\{\{\beta\}\}$ as a base.

If $D_\beta$ has no cluster point in $(\beta,\tau_\beta)$, then it
is closed and discrete in $(\beta,\tau_\beta)$. Let $\cB(\beta)$
be the family of all sets $C\cup\{\beta\}$ such that $C$ is a
cofinite subset of $D_\beta$. Let $\tau_\alpha$ be the topology on
$\alpha$ with weak base given by the families $\cB(\gamma)$,
$\gamma<\alpha$. In other words, $\tau_\alpha$ is the topology having
as a base the sets $U\in\tau_\beta$ as well as the sets $U$ such
that: $\beta\in U$, $U\cap\beta\in\tau_\beta$, and $U$ contains
some member of $\cB(\beta)$.

This completes the inductive construction. We take $X=\lambda$, with the topology $\tau_\lambda$.

For each $\alpha$, consider the space $(\alpha,\tau_\alpha)$ with
the weak base given by the families $\cB(\beta)$, $\beta<\alpha$,
defined in the construction. By definition, it is weakly locally
countable. We prove that $(\alpha,\tau_\alpha)$ is
weakly regular.

The construction is such that open sets remain open in all extensions,
and no new relative open sets are added to the already constructed spaces.
Thus, sets which are compact at some stage remain compact.
At each step, the added weak base elements are either singletons or convergent
sequences, and consequently are compact. It thus remains to prove that
the spaces are Hausdorff, and this is done by induction.
(By Lemma \ref{alephesH}, for spaces with a weak base consisting of countable compact sets,
being Hausdorff is equivalent to being weakly regular.)

\paragraph{$\alpha=\omega$:} this is evident.

\paragraph{$\alpha$ is a limit ordinal:}
If $\beta<\gamma<\alpha$ are distinct, by the induction hypothesis
they are separated by disjoint open sets from $\tau_{\gamma+1}$.

\paragraph{$\alpha=\beta+1$ is a successor ordinal:}
If $D_\beta$ has a cluster point in $(\beta,\tau_\beta)$, then
$\cB(\beta)=\{\{\beta\}\}$. For each $\gamma<\beta$, the open sets
$\beta$ and $\{\beta\}$ are disjoint and contain the points
$\gamma$ and $\beta$, respectively. By the induction hypothesis,
any two points in $\beta$ are separated by disjoint open sets.
Thus, $(\alpha,\tau_\alpha)$ is Hausdorff.

In the remaining case, where $D_\beta$ is closed discrete in
$(\beta,\tau_\beta)$, we prove that $(\alpha,\tau_\alpha)$ is
weakly regular. Let $\gamma\in \beta$. Then
$D_\beta\cup\{\gamma\}$ is discrete in $(\beta,\tau_\beta)$. Take
$U\in\tau_\beta$ such that $\gamma\in U$ and $U\cap
(D_\beta\sm\{\gamma\})=\emptyset$. Let $V\in\cB(\gamma)$ be a
subset of $U$. Then $V$ is disjoint from
$(D_\beta\sm\{\gamma\})\cup\{\beta\}\in\cB(\beta)$. Thus,
$(\beta,\tau_\beta)$ is weakly Hausdorff.

It remains to show that the elements of the weak base of
$(\alpha,\tau_\alpha)$ are closed. Let $B\in\cB(\beta)$.
$B\sm\{\beta\}$ is a cofinite subset of $D_\beta$, and thus
$\alpha\sm B =
\beta\sm(B\sm\{\beta\})\in\tau_\beta\sbst\tau_\alpha$. For
$\gamma<\beta$, let $B\in\cB(\gamma)$. As $B$ is compact and
$D_\beta$ is closed and discrete in $(\beta,\tau_\beta)$, $B\cap
D_\beta$ is finite. By the induction hypothesis, $B$ is closed in
$(\beta,\tau_\beta)$, and therefore $\beta\sm B\in\tau_\beta$, and
contains a cofinite subset of $D_\beta$. Thus, $\alpha\sm B =
(\beta\sm B)\cup\{\beta\}\in\tau_\alpha$.

This completes the proof that $X$ is Hausdorff.

Each nonempty subset $A$ of $X$ has an isolated point:
Take $\alpha=\min A$. Then $\alpha\in\alpha+1\in\tau_{\alpha+1}\sbst\tau_\lambda$,
and $\alpha+1$ is disjoint from $A\sm\{\alpha\}$.

$X$ has $\lambda$ many isolated points:
Let $\beta<\lambda$ be such that $D_\beta=\omega$.
For each $\alpha>\beta$, $\omega\cup\{\alpha\}$ is a countable
subset of $\lambda$ which is not a subset of $\beta$. Thus,
these sets are considered in (successor) stages later than $\beta+1$, where they
already have a cluster point. Consequently, in each of these $\lambda$ many
stages, a new isolated point is added.

$X$ is countably compact: Let $A$ be a countable subset of $\lambda$.
Let $\beta$ be such that $D_\beta=A$. Then either $A$ has
a cluster point in $(\beta,\tau_\beta)$ and thus in $X$ (since $\beta$ is open in $X$),
or else $A$ converges to $\beta$ (in $(\beta+1,\tau_{\beta+1})$
and therefore in $X$).

$X$ is sequentially compact:
This can be proved by quoting basic facts. Indeed,
$X$ is countably compact and has a countable weak base at each point.
Thus, $X$ is sequential, and therefore sequentially compact.

Alternatively, we can argue directly.
Let $A$ be a countable subset of $\lambda$.
Let $\beta$ be the minimal cluster point of $A$.
As $\beta$ is not isolated in $X$, $D_\beta$ converges to $\beta$.
In $(\tau_\beta,\beta)$, $A\cap\beta$ is closed, and therefore $U=\beta\sm A$ is open
in $X$. If $D_\beta\sm U$ is finite, then $U\cup\{\beta\}$ is an open neighborhood of
$\beta$, and $(U\cup\{\beta\})\cap A\sbst\{\beta\}$, in contradiction to $\beta$ being
a cluster point of $A$.
Thus, $D_\beta\sm U=D_\beta\cap A$ is infinite, and is therefore a subsequence of $A$
converging to $\beta$.
\end{proof}

\begin{proof}[Proof of Theorem \ref{main}]
Assume that $\lambda^\alephes=\lambda$, and let $X$ be as in Lemma \ref{mainlem}.
We may assume that $X\cap \lambda^*=\emptyset$.
We will define the required topology on $S=X\cup \lambda^*$.

Let $I$ be the set of isolated points of $X$, and let
$\{A_x : x \in I\}$ be maximal almost disjoint in $[\lambda^*]^\alephes$.
The open
sets in the topology of $S$ are those of the form
$$U\cup V,$$
where $U$ is open in $X$, $V$ is open in $\lambda^*$, and
for each $x\in I\cap U$, $A_x \sm V$ is finite.
(Note that for $U=\emptyset$, we get that each open
set in $\lambda^*$ is also open in $S$.)

Clearly, $|S|=\lambda$.

$S$ is Hausdorff: Consider any distinct $x_0,x_1\in X$ and
any distinct $y_0,y_1\in \lambda^*$. We will find disjoint open
sets $G_0,G_1$ such that $x_i,y_i\in G_i$, $i=0,1$.
$X$ is Hausdorff and locally countable.
Take disjoint countable open subsets $U_0,U_1$ of $X$
containing $x_0,x_1$, respectively.
Enumerate $I\cap (U_0\cup U_1)=\{x_n : n\in\N\}$.
For each $n$, define $\tilde A_{x_n} = A_{x_n}\sm\Union_{k<n} A_{x_k}$.
Note that $\tilde A_{x_n}$ is a cofinite subset of $A_{x_n}$.
By Lemma \ref{lambda*}, there are disjoint open sets $V_0,V_1$ in $\lambda^*$
such that
$$(\{y_i\}\cup\Union_{x\in I\cap U_i}\tilde A_x)\sm\{y_{1-i}\}\sbst V_i,$$
$i=0,1$.
The sets $G_i=U_i\cup V_i$, $i=0,1$, are as required.

$S$ is dense in itself: $\lambda^*$ is dense in itself, hence
(by the definition of open sets in $S$) each
$y\in \lambda^*$ is an accumulation point in $S$.
Let $x\in X$. If $x$ is an accumulation point in $X$,
then $x$ is an accumulation point in $S$.
Otherwise, $x$ is in the closure of $A_x$:
Indeed, the sets $\{x\}\cup V$ where $V$ is open
in $\lambda^*$ and $A_x\sm V$ is finite form a neighborhood
base at $x$, and therefore the elements of $A_x$ converge to $x$.

Each $Z\sbst S$ with $|Z|<\cf(\lambda)$ is scattered:

Case 1: $Z\sbst X$. As $X$ is scattered, we can find
$x\in Z$ isolated in $X$. Then $\{x\}\cup \lambda^*$ is a neighborhood
of $x$ in $S$ and is disjoint from $Z\sm\{x\}$.

Case 2: $Z\cap \lambda^*\neq\emptyset$.
$|Z\cap \lambda^*|<\cf(\lambda)\le\lambda=|I|$.
Thus, $Z\cap \lambda^*$ is discrete in $\lambda^*$.
As each open set in $\lambda^*$ is also open in $S$,
we are done.

$S$ is sequentially compact:
Assume that $Z\in [S]^\alephes$.
If $Z\cap X$ is infinite, then there are
$x\in X$ and a subsequence of $Z\cap X$ converging
to $x$. It is clear that the same subsequence will also
converge to $x$ in the topology of $S$.
And if not, then $Z\cap \lambda^*$ is infinite.
As $\{A_x : x \in I\}$ is maximal almost disjoint in $[\lambda^*]^\alephes$,
there is $x\in I$ such that $A_x\cap Z\cap \lambda^*$ is infinite.
Being a subsequence of $A_x$, it converges to $x$.

This completes the proof of Theorem \ref{main}.
\end{proof}

\subsection*{Acknowledgments}
We thank Istv\'an Juh\'asz for his substantial simplification
of our proof, and Peter Nyikos and Lyubomyr Zdomskyy for their useful suggestions.

Our research is partially supported by: The United States-Israel Binational Science Foundation grant number 2002323
(first author); and Koshland Center for Basic research (second author).

This is Publication 921 on Shelah's list.

\end{document}